\documentclass[10pt,]{article}
\usepackage{graphicx,amsthm}
\usepackage{amsfonts,amssymb,amscd,amsmath,enumerate,verbatim,latexsym,epsfig,amsfonts}

\newtheorem{Theorem}{Theorem}[section]
\newtheorem{Lemma}[Theorem]{Lemma}
\newtheorem{Corollary}[Theorem]{Corollary}
\newtheorem{Proposition}[Theorem]{Proposition}
\newtheorem{Remark}[Theorem]{Remark}

\begin{document}
\title{Locating-Dominating Sets and Identifying Codes of a Graph Associated to a Finite Vector Space}
\author{Muhammad Murtaza, Imran Javaid\thanks{Corresponding author}, Muhammad Fazil
\vspace{4mm}\\
\normalsize  Centre for Advanced Studies in Pure and Applied
Mathematics,\\\normalsize Bahauddin Zakariya University, Multan,
Pakistan\\\normalsize {\tt imran.javaid@bzu.edu.pk}}

\date{}
\maketitle

\begin{abstract}
In this paper, we investigate the problem of covering the vertices
of a graph associated to a finite vector space as introduced by Das
\cite{Das}, such that we can uniquely identify any vertex by
examining the vertices that cover it. We use locating-dominating
sets and identifying codes, which are closely related concepts for
this purpose. These sets consist of a dominating set of graph such
that every vertex is uniquely identified by its neighborhood within
the dominating sets. We find the location-domination number and the
identifying number of the graph and study the exchange property for
locating-dominating sets and identifying codes.
\end{abstract}




\section{Preliminaries}
The association of graphs to algebraic structures has become the
interesting research topic for the past few decades. See for
instance: commuting graphs for groups \cite{bates1,bates2,com},
power graphs for groups and semigroups
\cite{Cameron,Chakrabarty,Moghaddamfar}, zero divisor graph
associated to a commutative ring \cite{Anderson,Beck}. The
association of a graph and vector space has history back in 1958 by
Gould \cite{Gould}. Later, Chen \cite{chen} investigated on vector
spaces associated with a graph. Carvalho \cite{Car} has studied
vector space and the Petersen Graph. In the recent past, Manjula
\cite{Manjula} used vector spaces and made it possible to use
techniques of linear algebra in studying the graph. Intersection
graphs assigned to vector space were studied \cite{Rad,Talebi}. Das
\cite{Das} introduced a new graph structure, called non-zero
component graph on finite dimensional vector spaces. He showed that
the graph is connected and found its domination number and
independence number \cite{Das1}. He characterized the maximal
cliques in the graph and found the exact clique number, for some
particular cases \cite{Das1}. Das has also given some results on
size, edge-connectivity and the chromatic number of the graph
\cite{Das1}.

The covering code problem for a given graph involves finding a
minimum set of vertices whose neighborhoods uniquely overlap at any
given graph vertex. The problem has demonstrated its fundamental
nature through a wide variety of applications. Locating-dominating
sets were introduced by Slater \cite{slater1, slater3} and
identifying codes by Karpovsky et al. \cite{Karpov}.
Locating-dominating sets are very similar to identifying codes with
the subtle difference that only the vertices not in the
locating-dominating set are required to have unique identifying
sets. The decision problem for locating-dominating sets for directed
graphs has been shown to be an NP-complete problem \cite{char2}. A
considerable literature has been developed in this field (see
\cite{ber,char1,col,fin,hon,rall,slater1,slater2}). In \cite{cac1},
it was pointed out that each locating-dominating set is both
locating and dominating set. However, a set that is both locating
and dominating is not necessarily a locating-dominating set.

The initial application of locating-dominating sets and identifying
codes was fault-diagnosis in the maintenance of multiprocessor
systems \cite{Karpov}.
More recently, identifying codes and locating-dominating sets were
extended to applications for joint monitoring and routing in
wireless sensor networks \cite{Laifen} and environmental monitoring
\cite{berger}.


A natural question arises in reader's mind that how can we
distinguish the need of identifying codes or locating-dominating
sets for a system? A system, in which processors or sensors are able
to send the information about themselves and their neighbors, an
identifying code is necessary. However, the systems where the
sensors work without failure or if their only task is to test their
neighborhoods (not themselves) then we shall search for
locating-dominating sets. Moreover, the existence of identifying
codes is not always guaranteed in a graph (as we shall see in our
later discussion) and then a locating-dominating set is the next
best alternative.

In this paper, we study the locating-dominating sets and identifying
codes for the graph associated to finite vector space as defined in
\cite{Das}. Also, we find location-domination number and identifying
number of the graph and study the exchange property of the graph for
these graph invariants.

Now, we recall some definitions of graph theory which are necessary
for this article. We use $\Gamma$ to denote a connected graph with
vertex set $V(\Gamma)$ and edge set $E(\Gamma)$. The \emph{degree}
of the vertex $v$ in $\Gamma$, denoted by $deg(v)$, is the number of
edges to which $v$ belongs. The \emph{open neighborhood} of the
vertex $u$ of $\Gamma$ is $N(u)=\{v\in V(\Gamma):uv\in E(\Gamma)\}$
and the \emph{closed neighborhood} of $u$ is $N[u]=N(u)\cup \{u\}$.

 Formally, we define a locating-dominating set as: A set $L_D$ of vertices of $\Gamma$ is
called a locating-dominating set for $\Gamma$ if for every two
distinct vertices $u,v \in V(\Gamma)\setminus L_D$, we have
$\emptyset \neq N(u)\cap L_D\neq N(v)\cap L_D\neq\emptyset$. The
{\it location-domination number}, denoted by $\lambda(\Gamma)$, is
the minimum cardinality of a locating-dominating set of $\Gamma$.

An identifying code is a subset of vertices in a graph with the
property that the neighborhood of every vertex has a unique
intersection with the code. Formally it is defined as: A set $I_D$
is called an \emph{identifying code} for the graph $\Gamma$ if
$N[u]\cap I_D\ne N[v]\cap I_D$ for all $u,v\in V(\Gamma)$. The
cardinality of a smallest identifying code is called the
\emph{identifying number} of $\Gamma$ and we denote it by
$I(\Gamma)$.

Unlike identifying codes, every graph has a trivial
locating-dominating set, the entire set of vertices. On the other
hand, a graph may not be an identifying code, because if
 $N[u]=N[v]$ for some $u,v\in V(\Gamma)$, then clearly
 $V(\Gamma)$ is not an identifying code. Since an identifying
  code is also a locating-dominating set, therefore
\begin{equation}\label{Inequality Lamda and I}
 \lambda(\Gamma(\mathbb{V}))\le I(\Gamma(\mathbb{V})).
\end{equation}
Two vertices $u,v$ are \emph{adjacent twins} if $N[u]=N[v]$ and
\emph{non-adjacent twins} if $N(u)=N(v)$. If $u,v$ are adjacent or
non-adjacent twins, then $u,v$ are \emph{twins}. A set of vertices
$T$ is called a \emph{twin-set} if any two of its vertices are twins
\cite{Hernando}. By definition of twin vertices and twin-set, we
have the following straightforward  results:
\begin{Proposition}\label{fazf} Suppose that $u$, $v$ are twins in a connected
graph $\Gamma$ and $L_D$ is a locating-dominating set of  $\Gamma$,
then either $u$ or $v$ is in $L_D$. Moreover, if $u\in L_D$ and
$v\not\in L_D$, then $(L_D\setminus \{u\})\cup\{v\}$ is a
locating-dominating set of $\Gamma$.
\end{Proposition}
\begin{Proposition}\label{rem3.3}
Let $T$ be a twin-set of order $m\geq 2$ in a connected graph
$\Gamma$. Then, every locating-dominating set $L_D$ of $\Gamma$
contains at least $m-1$ vertices of $T$.
\end{Proposition}

\subsection{Non-Zero Component Graph} Let $\mathbb{V}$ be a vector
space over a field $\mathbb{F}$ with a basis $\{b_{1},
b_{2},\ldots,b_{n}\}$. A vector $v\in \mathbb{V}$ is expressed
uniquely as a linear combination of the form $v
=c_{1}b_{1}+c_{2}b_{2}+\cdots+c_{n}b_{n}$. A \emph{non-zero
component graph}, denoted by $\Gamma(\mathbb{V})$, can be associated
with a finite dimensional vector space in the following way: the
vertex set of the graph $\Gamma(\mathbb{V})$ consists of the
non-zero vectors and two vertices are joined by an edge if they
share at least one $b_{i}$ with non-zero coefficient in their unique
linear combination with respect to $\{b_{1}, b_{2},\ldots,b_{n}\}$
\cite{Das}. It is proved in~\cite{Das} that $\Gamma(\mathbb{V})$ is
independent of the choice of basis, i.e., isomorphic non-zero
component graphs are obtained for two different bases.

\begin{figure}[h!]
        \centerline
        {\includegraphics[width=4cm]{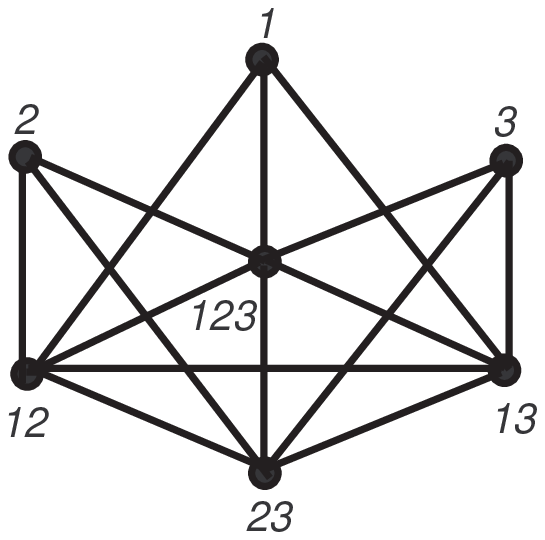}}\label{splitcycle}
      \caption{$dim(\mathbb{V})=3$; $\mathbb{F}= Z_2=\{\overline{0},\overline{1}\}$.}
\end{figure}

\begin{Theorem}\emph{\cite{Das1}}\label{order}
If $\mathbb{V}$ be an $n$-dimensional vector space over a finite
field $\mathbb{F}$ with $q$ elements, then the order of
$\Gamma(\mathbb{V})$ is $q^n-1$ and the size of $\Gamma(\mathbb{V})$
is
  $$\frac{q^{2n}-q^n+1-(2q-1)^n}{2}.$$
\end{Theorem}

\begin{Theorem}\emph{\cite{Das}}\label{degree} Let $\mathbb{V}$ be an $n$-dimensional
 vector space over a finite field $F$ with
$q$ elements and $\Gamma(\mathbb{V})$ be its associated graph with
respect to a basis $\{b_1, b_2, . . . , b_n\}$, then a vertex having
$s$ non-zero coefficients in its unique linear
combination of basis vector has degree 
$(q^{s}- 1)q^{n-s}-1.$
\end{Theorem}

\section{Locating-Dominating Sets and Identifying Codes of Non-Zero
 Component Graph}
 In this section, we study the location-domination
number of non-zero component graph $\Gamma(\mathbb{V}).$

We partition the vertex set of $\Gamma(\mathbb{V})$ into $n$ classes
$T_i$, where $T_i=\{v\in \mathbb{V}: v$ is a linear combination of
basis vectors with $i$ non-zero coefficients$\}$. For example, if
$n=3$ and $q=2$, then $T_2=\{b_1+b_2,b_2+b_3,b_1+b_3\}$.
\begin{Lemma}\label{CarNbr}
Let $\mathbb{V}$ be a vector space of dimension $n$ over a field
$\mathbb{F}$ of $2$ elements. If $v\in T_s$ for $s$ $(1\le s \le
n)$, then for $r$ $(1\le r \le n)$
$$|N(v)\cap T_r|=\left\{
\begin{array}{ll}
{n\choose r}-{{n-s}\choose r}-1\hskip 0.8cm if \hskip 1cm r\le n-s\,\,\, \mbox{and}\,\,\,r=s  \\
{n\choose r}-{{n-s}\choose r}\hskip 1.4cm if \hskip 1cm r\le n-s\,\,\, \mbox{and}\,\,\,r\ne s \\
{n\choose r}-1\hskip 2.1cm if \hskip 1cm n-s<r\le n\,\,\, \mbox{and}\,\,\,r=s \\
{n\choose r}\hskip 2.7cm if \hskip 1cm n-s<r\le n\,\,\, \mbox{and}\,\,\,r\ne s.\\
\end{array}
\right.$$
\end{Lemma}
\begin{proof} We consider the following cases for $r$:
\begin{enumerate}
\item If $r \leq n-s$, then ${{n-s}\choose r}$ elements of $T_r$
have $s$ zero coefficients in their unique linear combination of
basis vectors for those $s$ basis vectors which have the non-zero
coefficients in the unique linear combination of $v$, and hence
these elements of $T_r$ are not adjacent to $v$. Thus, $|N(v)\cap
T_r|= {n\choose r}-{{n-s}\choose r}$ or ${n\choose r}-{{n-s}\choose
r}-1$ according as $r\neq s$ or $r=s$, respectively.
\item If $r>n-s$, then each element of $T_r$ will have at least one
non-zero coefficient in its unique linear combination of basis
vectors for those $s$ basis vectors which have the non-zero
coefficients in the unique linear combination of $v$, and hence $v$
is adjacent to all elements of $T_r$. Thus, $|N(v)\cap T_r|=
{n\choose r}$ or ${{n}\choose r}-1$ according as $r\neq s$ or $r=s$,
respectively.
\end{enumerate}
\end{proof}
Let $v\in T_s$, then it can be seen from Lemma \ref{CarNbr} that
$deg(v)=[\sum \limits_{r=1}^n |N(v)\cap T_r|]-1=\sum
\limits_{r=1}^{n-s}[{n\choose r}-{{n-s}\choose r}]+\sum
\limits_{r=n-s+1}^{n}{n\choose r}-1 =(2^s-1)2^{n-s}-1$ which is
consistent with Theorem \ref{degree} for $q=2$.

\begin{Remark}
Let $\mathbb{V}$ be a vector space of dimension $n$ over a field
$\mathbb{F}$ of $2$ elements. If $v\in T_s$ for $s$ $(1\le s \le
n)$, then $deg(v)=(2^s-1)2^{n-s}-1$.
\end{Remark}

\begin{Lemma}\label{T2 Intersections}
Let $\mathbb{V}$ be a vector space of dimension $n\ge 4$ over a
field $\mathbb{F}$ of $2$ elements. If $u,v\in
V(\Gamma(\mathbb{V}))\setminus T_{n-1}$, then $N(u)\cap T_2\ne
N(v)\cap T_2$.
\end{Lemma}

\begin{proof}
Since $u\in T_r$ and $v\in T_s$ for some $1\le r,s \le n$ $(r,s\ne
n-1)$, therefore $u$ has $r$ non-zero coefficients in its unique
linear combination of basis vectors $B=\{b_1,b_2,...,b_n\}$. Let
$B_u\subseteq B$ and $B_v\subseteq B$ is the set of those basis
vectors which has non-zero coefficients in the unique linear
combination of basis vectors for $u$ and $v$ respectively. Then $u$
is not adjacent to ${n-r\choose 2}$ elements of $T_2$ which have
exactly two non-zero coefficients of basis vectors in $B_u$ and zero
coefficients of basis vectors in $B\setminus B_u$. Since $u\not\in
T_{n-1}$, therefore such elements exist in $T_2$ which has exactly
two non-zero coefficients of basis vectors of $B_u$. Thus, $N(u)\cap
T_2=T_2\setminus \{$two element sum of basis vectors in $B\setminus
B_r\}$. Since $u\ne v$, therefore $B_u\ne B_v$, and hence $N(u)\cap
T_2\ne N(v)\cap T_2$.
\end{proof}

An immediate consequence of Lemma \ref{T2 Intersections} is that the
set $T_2\cup T_{n-1}$ forms a locating-dominating set for
$\Gamma(\mathbb{V})$ for a vector space $\mathbb{V}$ of dimension
$n\ge 4$ over a field of 2 elements.

Since elements of $T_{n-1}$ have non-zero coefficients for $n-1$
basis vectors, therefore we use the notation
$u_j=\sum\limits_{i=1}^n b_i-b_j$ in proof of Lemma \ref{Card of
LD1} for the element of $T_{n-1}$ which has zero coefficient for the
basis vector $b_j$. Also, $N[u_j]=V(\Gamma(\mathbb{V}))\setminus
\{b_j\}$, therefore two elements $u_i,u_j\in T_{n-1}$ have same
neighbors in $\Gamma(\mathbb{V})$ except the elements $b_i$ and
$b_j$ of $T_1$.

\begin{Lemma}\label{Card of LD1}
Let $\mathbb{V}$ be a vector space of dimension $n\ge 3$ over a
field $\mathbb{F}$ of $2$ elements. Let $L_D$ be a
locating-dominating set for $\Gamma(\mathbb{V})$ and $|L_D\cap T_1|=
s$ \\
\emph{(a)} If $0\le s \le n-2$, then $|L_D\cap T_{n-1}|\ge n-s$.\\
\emph{(b)} If $s=n-1$, then $|L_D\cap \{T_n\cup T_{n-1}\}|\ge 1$.
\end{Lemma}

\begin{proof}
Without loss of generality assume that $L_D\cap
T_1=\{b_1,b_2,...,b_s\}$.\\
(a) Let $u_i,u_j\in T_{n-1}$ for $s+1\le i \ne j \le n$ be two
distinct elements of $T_{n-1}$, then $N(u_i)\cap \{L_D\cap
T_1\}=N(u_j)\cap \{L_D\cap T_1\}=\emptyset$. Since $u_i$ and $u_j$
have different neighbors only in
$\{b_{s+1},b_{s+2},...,b_n\}\subseteq T_1$ which is not subset of
$L_D$, therefore these $n-s$ elements of $T_{n-1}$ must belong to
$L_D$. Hence, $|L_D\cap T_{n-1}|\ge n-s$. \\
(b) Let $u_n\in T_{n-1}$ and $v\in T_n$, then $N(u_n)\cap \{L_D\cap
T_1\}=N(v)\cap \{L_D\cap T_1\}=\emptyset$. Since $u_n$ and $v$ have
only one different neighbor $b_n\in T_1$ which is not in $L_D$,
therefore either $u_n$ or $v$ must belong to $L_D$.
\end{proof}

\begin{Corollary}\label{lower bound of L_D}
Let $\mathbb{V}$ be a vector space of dimension $n\ge 3$ over a
field $\mathbb{F}$ of $2$ elements. Let $L_D$ be a
locating-dominating set for $\Gamma(\mathbb{V})$, then $|L_D|\ge n$.
\end{Corollary}
\begin{proof}
If $0\le s \le n-2$, then $|L_D\cap \{T_1\cup T_{n-1}\}|\ge s+n-s=n$
by Lemma \ref{Card of LD1}(a). If $s=n-1$, then $|L_D\cap \{T_1\cup
T_{n-1}\cup T_n\}|\ge n-1+1=n$ by Lemma \ref{Card of LD1}(b). If
$s=n$, then clearly $|L_D\cap T_1|=n$.
\end{proof}
%
Since $\lambda(P_3)=2$ where $P_3$ is the path graph of order 3,
therefore we have the following proposition.
\begin{Proposition}
Let $\mathbb{V}$ be a vector space of dimension $2$ over a field
$\mathbb{F}$ of $2$ elements, then $\lambda(\Gamma(\mathbb{V}))= 2$.
\end{Proposition}
Let $\mathbb{V}$ be a vector space of dimension $n$ and $q\ge 3$.
Then the class $T_i$ for each $i$ $(1\le i \le n)$ has $n\choose i$
twin subsets of vertices of $\Gamma(\mathbb{V})$ and each of these
twin subsets has the cardinality $(q-1)^i$. We use the notation
$T_{i_{k}}$ where $1\le k \le {n\choose i}$ to denote the $k$th twin
set in the class $T_i$.
\begin{Theorem}\label{fazz}
Let $\mathbb{V}$ be a vector space over a field $\mathbb{F}$ of $q$ elements
with $\{b_{1},b_{2},\ldots,b_{n}\}$ as basis:\\
a) If $q=2$ and $n\geq 3$, then $ \lambda(\Gamma(\mathbb{V}))=n$.

b) If $q\geq 3$, then $\lambda({\Gamma(\mathbb{V})})= \sum
\limits_{i=1}^n\binom{n}{i}((q-1)^i-1)$.
\end{Theorem}
\begin{proof}
a) For $q=2$ and $n\ge 3$, we first prove that $T_1$ is a
locating-dominating set for $\Gamma(\mathbb{V})$. Let $u,v\in
V(\Gamma(\mathbb{V}))\setminus T_1$. If $u,v\in T_s$ for some $s$
 when $2\le s \le n-1$, then both $u$ and $v$ have $s$ non-zero
coefficients in their linear combinations of basis vectors. Since
$u\ne v$ and $s<n$, therefore $\emptyset \ne N(u)\cap T_1\ne
N(v)\cap T_1\ne \emptyset$. If $u \in T _r$ and $v \in T_s$ for some
$r,s$ $2\le r\ne s \le n$, then $|N(u)\cap T_1|\ne |N(v)\cap T_1|$
by Lemma \ref{CarNbr}, and hence $\emptyset \ne N(u)\cap T_1\ne
N(v)\cap T_1\ne \emptyset$. Thus, $T_1$ is a locating dominating set
for $\Gamma(\mathbb{V})$. Hence, $\lambda(\Gamma(\mathbb{V}))\le n$.
Also $\lambda(\Gamma(\mathbb{V}))\ge n$ by Corollary \ref{lower
bound of L_D}.

b) If $q\geq3$, then from Proposition \ref{rem3.3}, a minimal
locating-dominating set of $\Gamma(\mathbb{V})$ contains at least
$(q-1)^i-1$ vertices from $T_{i_k}$ for each $i$ ($1\le i \le n$)
and each $k$ ($1\le k \le {n\choose i}$), and hence
$\lambda({\Gamma(\mathbb{V})})\geq \sum
\limits_{i=1}^n\binom{n}{i}((q-1)^i-1)$. Moreover, a subset of
$\Gamma(\mathbb{V})$ of cardinality greater than $\sum
\limits_{i=1}^n\binom{n}{i}((q-1)^i-1)$ has all the vertices of at
least one twin subset $T_{i_k}$. Thus, from Proposition \ref{fazf},
a locating-dominating set of cardinality greater than $\sum
\limits_{i=1}^n\binom{n}{i}[(q-1)^i-1]$ is not a minimal
locating-dominating set, and hence
$\lambda({\Gamma(\mathbb{V})})\leq \sum
\limits_{i=1}^n\binom{n}{i}[(q-1)^i-1]$.
\end{proof}
Since $I(P_3)=2$, therefore we have the following proposition.
\begin{Proposition}
Let $\mathbb{V}$ be a vector space of dimension $2$ over a field
$\mathbb{F}$ of $2$ elements, then $I(\Gamma(\mathbb{V}))=2$.
\end{Proposition}
 Following theorem gives the identifying number of $\Gamma(\mathbb{V})$.
\begin{Theorem}
Let $\mathbb{V}$ be a finite vector space over a field $\mathbb{F}$
of $2$ elements, then $I(\Gamma(\mathbb{V}))=n$.
\end{Theorem}
\begin{proof}
For $n\ge 3$ and $q=2$, by Theorem \ref{fazz}(a) and inequality
(\ref{Inequality Lamda and I}), $I(\Gamma(\mathbb{V})\ge n$. Note
that, $T_1$ is an identifying code for $\Gamma(\mathbb{V})$ because
for each vertex say $u\in V(\Gamma(\mathbb{V}))$, $N[u]\cap T_1$ is
the set of all those elements of $T_1$ which has non-zero
coefficients in the representation of $u$ as the unique linear
combination of basis vectors. Thus, for any two distinct elements
$u,v\in V(\Gamma(\mathbb{V}))$, $N[u]\cap T_1$ and $N[v]\cap T_1$
are distinct. Hence, $I(\Gamma(\mathbb{V}))\le n$.
\end{proof}
Let $\mathbb{V}$ be a finite vector space and $q\ge 3$, then
$\Gamma(\mathbb{V})$ has twin sets $T_{i_k}$ ($1\le i \le n$) ($1\le
k \le {n\choose i}$) and each of these twin subset has adjacent
twins, therefore identifying code for $\Gamma(\mathbb{V})$ does not
exist. Thus, we have following remark.
\begin{Remark}
Let $\mathbb{V}$ be a vector space of dimension $n\ge 3$ and $q\ge
3$, then identifying code for $\Gamma(\mathbb{V})$ does not exist.
\end{Remark}
\begin{Lemma}\label{T1 Unique Identifying Code}
Let $\mathbb{V}$ be a vector space of dimension $n\ge 3$ and $q=2$,
then $T_1$ is the only minimal identifying code for
$\Gamma(\mathbb{V})$.
\end{Lemma}
\begin{proof}
Suppose on contrary $I_D'$ be another minimal identifying code of
$\Gamma(\mathbb{V})$, then there exist at least one element say $b_r
\in T_1$ such that $b_r\not\in I_D'$ (because otherwise $T_1\subset
I_D'$). Take two elements $u_r\in T_{n-1}$ (using same notation as
in proof of Lemma \ref{Card of LD1}) and $w\in T_n$. Since
$N[w]=V(\Gamma(\mathbb{V}))$ and
$N[u_r]=V(\Gamma(\mathbb{V}))\setminus \{b_r\}$, therefore $N[w]\cap
I_D'=N[u_r]\cap I_D'\ne \emptyset$, a contradiction.
\end{proof}
\subsection{Exchange Property} Locating-dominating sets are said to
have \emph{the exchange property} in a graph $\Gamma$ if whenever
$L_{D_1}$ and $L_{D_2}$ are minimal locating-dominating sets for
$\Gamma$ and $u_1\in L_{D_1}$, then there exists $u_2\in L_{D_2}$ so
that $(L_{D_2}\setminus\{u_2\})\cup\{u_1\}$ is also a minimal
locating-dominating set. If a graph $\Gamma$ has the exchange
property, then every minimal locating-dominating set for $\Gamma$
has the same number of vertices. To show that the exchange property
does not hold in a graph, it is sufficient to show that there exist
two minimal locating-dominating of different cardinalities. However,
the condition is not necessary, i.e., the exchange property does not
hold and, hence, does not imply that there are locating-dominating
sets of different cardinalities.
\begin{Lemma}\label{exchange property}
For $q= 2$ and $n> 3$, the exchange property does not hold for
locating-dominating sets in graph $\Gamma(\mathbb{V})$.
\end{Lemma}
\begin{proof}
For $n=4$, the exchange property does not hold because $T_1$ and
 $\{b_1+b_4,b_2+b_4,b_3+b_4\} \cup T_{3}$ are minimal
locating-dominating sets of different cardinalities.\\
For $n\ge 5$, $T_1$ and $T_2\cup T_{n-1}$ are two
locating-dominating sets of cardinalities $n$ and ${n\choose 2}+n$
by Lemma \ref{T2 Intersections}. For notational convenience, we use
$A=T_2\cup T_{n-1}$. We will prove that $A$ is a minimal
locating-dominating set of $\Gamma(\mathbb{V})$. Let $u\in A$ and
$w\in T_n$. There are two possible cases for $u$.
\begin{enumerate}
\item If $u\in T_2$, then $u$ has exactly two non-zero coefficients
in its unique linear combination of basis vectors, say these vectors
set as $B_u$. Choose an element in $v\in T_{n-2}$ such that $v$ has
exactly $n-2$ non-zero coefficients in the unique linear combination
of basis vectors in $B\setminus B_u$. Then $N(v)\cap A\setminus
\{u\}=N(w)\cap A\setminus \{u\}$. Thus, $A\setminus \{u\}$ is not
locating-dominating.
\item If $u\in T_{n-1}$, then $N(u)\cap A\setminus \{u\}=N(w)\cap
A\setminus \{u\}$.
\end{enumerate}
Thus, $T_2\cup T_{n-1}$ is a minimal locating-dominating set. Hence,
exchange property does not hold for locating-dominating sets in
graph $\Gamma(\mathbb{V})$.
\end{proof}

In the proof of Lemma \ref{exchange property q greater than 3}, we
use the same notation $T_{i_k}$ for the $k$th twin set of class
$T_i$ as we have used in the proof of Theorem \ref{fazz}(b).
\begin{Lemma}\label{exchange property q greater than 3}
For $q\geq 3$, the exchange property holds for locating-dominating
sets in graph $\Gamma(\mathbb{V})$.
\end{Lemma}
\begin{proof}
Since there are $(q-1)^i$ choices for removing one vertex from a
twin set $T_{i_k}$ of cardinality $(q-1)^i$, therefore there are
$\prod \limits_{i=1}^n\binom{n}{i}(q-1)^i$ minimal
locating-dominating sets in $\Gamma(\mathbb{V})$. Let $L_{D_1}\neq
L_{D_2}$ be two such minimal locating-dominating sets. Let $u_{1}\in
L_{D_1}$, we further assume that $u_1\not\in L_{D_2}$ (for otherwise
$(L_{D_2}\setminus \{u_{1}\})\cup \{u_{1}\}$ is, obviously, a
minimal locating-dominating set of $\Gamma(\mathbb{V})$). Also,
$u_{1}\in T_{i_k}$ for some $i$ ($1\leq i\leq n$) and some $k$
 ($1\le k \le {n\choose i}$). Since $u_1\in \{L_{D_1}\cap T_{i_k}\}\setminus \{L_{D_2}\cap
 T_{i_k}\}$ and $L_{D_1}$ and $L_{D_2}$ are minimal, therefore
 there exists an element $u_2\in \{L_{D_2}\cap T_{i_k}\}\setminus \{L_{D_1}\cap
 T_{i_k}\}$. Since both $u_1$ and $u_2$ belong to the same twin set
 $T_{i_k}$, therefore by Proposition \ref{fazf} $(L_{D_2}\setminus
  \{u_{2}\})\cup \{u_{1}\}$ is a minimal locating-dominating set
  of $\Gamma(\mathbb{V})$. Hence, exchange property holds in $\Gamma(\mathbb{V})$
\end{proof}

From Lemma \ref{T1 Unique Identifying Code} we have the following
remark.
\begin{Remark}
Let $\mathbb{V}$ be a vector space of dimension $n\ge 3$ and $q=2$,
then exchange property holds for identifying code holds in
$\Gamma(\mathbb{V})$.
\end{Remark}
From Lemma \ref{exchange property} we have the following remark.
\begin{Remark}
The locating-dominating sets does not have the exchange property for
all graphs.
\end{Remark}

\end{document}